\documentclass[11pt]{article}

\usepackage[T1]{fontenc}
\usepackage{lmodern}
\usepackage{amsmath,amssymb,amsthm}
\usepackage[margin=0.9in]{geometry}
\usepackage{cite}
\usepackage[hidelinks]{hyperref}
\numberwithin{equation}{section}

\hypersetup{
  pdftitle={The Wang Transform Inequality},
  pdfauthor={Yair Lavi}
}

\newtheorem{theorem}{Theorem}[section]
\newtheorem{lemma}[theorem]{Lemma}
\title{The Wang Transform Inequality}
\author{Yair Lavi}
\date{24 September 2026}

\begin{document}
\maketitle

\begin{abstract}
We prove the Wang transform inequality in every order: if $n$
is a positive integer, $A$ is a nonnegative $n\times n$ matrix whose row
and column sums equal one, and $J_n$ has every entry $1/n$, then
\[
\operatorname{per}A\ge
\operatorname{per}\!\left(\frac{nJ_n+A}{n+1}\right).
\]
Equality holds exactly when $A=J_n$. We also obtain an explicit positive
quadratic gap for every $n\ge2$.
\end{abstract}



\section{Introduction and main results}

For a positive integer $n$, let $[n]=\{1,\ldots,n\}$ and let $S_n$ be its
permutation group. The \emph{permanent} of an $n\times n$ matrix
$D=(d_{ij})$ is
\begin{equation}\label{eq:permanent}
\operatorname{per}D=\sum_{\pi\in S_n}\prod_{i=1}^n d_{i,\pi(i)}.
\end{equation}
A real matrix is \emph{doubly stochastic} if its entries are nonnegative and
every row and column sum equals one. Write $\Omega_n$ for the set of such
$n\times n$ matrices and $J_n$ for the matrix with every entry $1/n$.
Wang's transform is the map $W_n:\Omega_n\to\Omega_n$ defined by
\begin{equation}\label{eq:transform}
W_n(A)=\frac{nJ_n+A}{n+1}.
\end{equation}

Wang proposed inequality~\eqref{eq:wang}~\cite{wang1977}. Minc included it
in his catalogue of open problems~\cite[Section~8.4]{minc1978} and restated
it as Conjecture~4 in his 1983 survey~\cite[Section~5]{minc1983}.
Cheon and Wanless retained it as Conjecture~4 in their later
update~\cite{cheonwanless2005}. Lih and Wang proved a stronger monotonicity
result in order three~\cite{lihwang1981}. Chang and Foregger independently
proved order four~\cite{cheonwanless2005,chang1988,foregger1988}. Chang
handled matrices outside a neighbourhood of $J_n$~\cite{chang1983}, and
Hwang handled matrices that become a nontrivial block-diagonal direct sum
after row and column permutations~\cite{cheonwanless2005,hwang1989}.

\begin{theorem}\label{thm:wang}
For every positive integer $n$ and every $A\in\Omega_n$, the following
inequality holds, with equality if and only if $A=J_n$:
\begin{equation}\label{eq:wang}
\operatorname{per}A\ge\operatorname{per}W_n(A).
\end{equation}
\end{theorem}

The proof compares a lower bound for $\operatorname{per}A$ with an
upper bound for $\operatorname{per}W_n(A)$.

For a real matrix $X=(x_{ij})$, its Frobenius norm is defined by
$\|X\|_F^2=\sum_{i,j}x_{ij}^2$. Write $r^2=\|A-J_n\|_F^2$, and put
$\gamma_n=n!/n^n$ for every $n\ge1$. For $n\ge2$, define
\begin{equation}\label{eq:constants}
\begin{aligned}
C_n&=\frac{\displaystyle\sum_{k=0}^n n!/k!}{(n+1)^n},&
a_n&=\frac1{2n}+\frac1{3n^2},&
b_n&=\frac{C_n/\gamma_n-1}{n-1}.
\end{aligned}
\end{equation}

\begin{theorem}\label{thm:bounds}
For every integer $n\ge2$ and every $A\in\Omega_n$,
\begin{equation}\label{eq:bounds}
\begin{aligned}
\operatorname{per}A&\ge\gamma_n e^{a_n r^2},&
\operatorname{per}W_n(A)&\le\gamma_n(1+b_n r^2).
\end{aligned}
\end{equation}
\end{theorem}

We will show that $0<b_n<1/(2n)<a_n$ for every $n\ge2$.
The bounds in~\eqref{eq:bounds} then give the global gap
\begin{equation}\label{eq:gap}
\begin{aligned}
\operatorname{per}A-\operatorname{per}W_n(A)
&\ge\gamma_n\bigl(e^{a_n r^2}-1-b_n r^2\bigr)\\
&\ge\gamma_n(a_n-b_n)r^2.
\end{aligned}
\end{equation}
The positive coefficient in the last bound is explicit; no optimality
is claimed. We proved the lower estimate and the centered subpermanent
bound used here in~\cite{lavi2026marcus}. Section~2 states these inputs
precisely. Section~3 derives the Wang-specific upper estimate, and
Section~4 compares the constants and completes the proof.

\section{Input estimates}

We proved the lower estimate in~\eqref{eq:bounds} for every $n\ge2$ and every
$A\in\Omega_n$ in~\cite[Theorem~1.2, first inequality]{lavi2026marcus},
including matrices on the boundary. We use that result without repeating
its proof.

For a real $n\times n$ matrix $X$ and row and column sets $R,S$ of equal
size, write $X[R,S]$ for the corresponding submatrix. Let
\[
\sigma_s(X)=\sum_{\substack{R,S\subset[n]\\|R|=|S|=s}}
\operatorname{per}X[R,S],\qquad \sigma_0(X)=1.
\]
A matrix has \emph{zero margins} if every row and column sum is zero.
Let $I_n$ be the identity matrix, let $Q=I_n-J_n$, and set
$\|X\|_{\rm op}=\sup_{v\ne0}\|Xv\|_2/\|v\|_2$.

\begin{lemma}[Centered subpermanents; \cite{lavi2026marcus}, Lemma~3.1]
\label{lem:centered}
For every integer $n\ge2$, every real $n\times n$ matrix $X$ with zero
margins and $\|X\|_{\rm op}\le1$, and every $2\le s\le n$,
\begin{equation}\label{eq:centered}
|\sigma_s(X)|\le
\frac{\sigma_s(Q)}{n-1}\|X\|_F^2.
\end{equation}
\end{lemma}

We proved this estimate in~\cite[Lemma~3.1]{lavi2026marcus}. Its proof
also establishes $\sigma_s(Q)\ge0$ for $2\le s\le n$
\cite[equation~(3.5)]{lavi2026marcus}. As verified just after that lemma
in~\cite{lavi2026marcus}, if $A\in\Omega_n$ then $X=A-J_n$ has zero
margins and $\|X\|_{\rm op}\le1$, so~\eqref{eq:centered} applies.

\section{The transformed upper bound}

Let $X=A-J_n$. In~\cite[equation~(4.1)]{lavi2026marcus} we proved the
\emph{centered expansion}, valid for every real $t$:
\begin{equation}\label{eq:expansion}
\operatorname{per}(J_n+tX)
=\sum_{s=0}^n\frac{(n-s)!}{n^{n-s}}
t^s\sigma_s(X).
\end{equation}
Here $\sigma_1(X)=0$ because the entries of $X$ sum to zero, and
the $s=0$ term equals $\gamma_n$. Substitute $t=1/(n+1)$
in~\eqref{eq:expansion}, apply Lemma~\ref{lem:centered} term by term,
and use $r^2=\|X\|_F^2$:
\begin{equation}\label{eq:upper-sum}
\operatorname{per}W_n(A)
\le\gamma_n+\frac{r^2}{n-1}
\sum_{s=2}^n\frac{(n-s)!}{n^{n-s}(n+1)^s}\sigma_s(Q).
\end{equation}
The sum is exactly $\operatorname{per}W_n(I_n)-\gamma_n$
by~\eqref{eq:expansion}, since $I_n-J_n=Q$ and $\sigma_1(Q)=0$.
Let $\mathbf1=(1,\ldots,1)^{\mathsf T}$. Then
\begin{equation}\label{eq:identity-permanent}
\begin{aligned}
W_n(I_n)&=\frac{\mathbf1\mathbf1^{\mathsf T}+I_n}{n+1},\\
\operatorname{per}(\mathbf1\mathbf1^{\mathsf T}+I_n)
&=\sum_{k=0}^n\binom nk(n-k)!
=\sum_{k=0}^n\frac{n!}{k!}.
\end{aligned}
\end{equation}
The first sum counts the terms obtained by choosing identity entries
in a subset of positions before taking the permanent of the remaining
all-ones matrix. Hence $\operatorname{per}W_n(I_n)=C_n$, and
\eqref{eq:upper-sum} becomes
\begin{equation}\label{eq:upper}
\operatorname{per}W_n(A)
\le\gamma_n+\frac{C_n-\gamma_n}{n-1}r^2
=\gamma_n(1+b_nr^2),
\end{equation}
proving the second half of Theorem~\ref{thm:bounds}.

\section{The constant comparison and final proof}

\subsection{Comparison of the constants}

It remains to show that the two coefficients in~\eqref{eq:bounds} separate
strictly. The sum in~\eqref{eq:identity-permanent} is the integer
$\lfloor e n!\rfloor$: the omitted tail of
$e n!=\sum_{k=0}^{\infty}n!/k!$ is strictly between zero and one. Indeed,
for $k\ge n+1$ its first term is $1/(n+1)$ and each later term is at most
$1/(n+1)$ times its predecessor. Thus
\begin{equation}\label{eq:en-factorial}
C_n=\frac{\lfloor e n!\rfloor}{(n+1)^n}
\quad\text{and}\quad
\frac{C_n}{\gamma_n}
<e\left(\frac n{n+1}\right)^n.
\end{equation}
For $x>0$, the elementary logarithm inequalities
$\log(1+x)>x-x^2/2$ and $\log(1+x)>x/(1+x)$ give
\begin{equation}\label{eq:logarithms}
1-n\log(1+1/n)<\frac1{2n}
<\log\left(1+\frac1{2(n-1)}\right).
\end{equation}
For the second comparison, apply the second inequality at
$x=1/[2(n-1)]$ and observe that
$x/(1+x)=1/(2n-1)>1/(2n)$. Taking exponentials
in~\eqref{eq:en-factorial}--\eqref{eq:logarithms} shows
\begin{equation}\label{eq:b-bound}
(n-1)\left(\frac{C_n}{\gamma_n}-1\right)<\frac12.
\end{equation}
This estimates $b_n$ from above. For its positivity, we proved
$\sigma_s(Q)\ge0$ for $2\le s\le n$ in
\cite[equation~(3.5)]{lavi2026marcus}. Also, by the zero-line-sum identity
$\sigma_2(Y)=\|Y\|_F^2/2$, valid for every real matrix $Y$ with zero line
sums,
\begin{equation}\label{eq:sigma2}
\sigma_2(Q)=\frac{\|Q\|_F^2}{2}=\frac{n-1}{2}>0.
\end{equation}
To see the identity, expand $\sigma_2(Y)$ and sum over two
distinct rows and columns; after fixing one entry $y_{ij}$,
the sum of entries in the complementary rows and columns is
$y_{ij}$ because all line sums vanish. Therefore the sum in
\eqref{eq:upper-sum} with $Q$ is strictly positive, so $C_n>\gamma_n$ and
$b_n>0$.

If $n\ge3$, \eqref{eq:b-bound} and $(n-1)^2>n$ give
\begin{equation}\label{eq:constant-comparison}
b_n<\frac1{2(n-1)^2}<\frac1{2n}<a_n.
\end{equation}
For $n=2$, direct substitution gives $b_2=1/9<1/4<a_2=1/3$.
Thus $0<b_n<1/(2n)<a_n$ for every $n\ge2$.

The strict comparison makes equality analysis of the two individual
estimates unnecessary: their combination separates the permanents
whenever $A\ne J_n$.

\subsection{Completion of the proof}

\begin{proof}[Proof of Theorem~\ref{thm:wang}]
Let $n\ge2$. Theorem~\ref{thm:bounds} and $e^y\ge1+y$ give
\[
\begin{aligned}
\operatorname{per}A-\operatorname{per}W_n(A)
&\ge\gamma_n\bigl(e^{a_nr^2}-1-b_nr^2\bigr)\\
&\ge\gamma_n(a_n-b_n)r^2,
\end{aligned}
\]
which proves~\eqref{eq:gap}. The comparison above gives $a_n>b_n$,
so the last expression is strictly positive when $A\ne J_n$. At
$A=J_n$, both permanents equal $\gamma_n$. For $n=1$, the only doubly
stochastic matrix is $A=J_1=[1]$; it is fixed by $W_1$, and both
permanents equal one. This proves~\eqref{eq:wang} and its equality claim.
\end{proof}

\section*{Acknowledgments}

The proof of this conjecture was carried out by GPT-5.6-sol, GPT-6 Astra,
and Claude Fable 5, under the guidance of the author. The author has reviewed the
resulting proof arguments. Responsibility for the final text rests with the
author.

\end{document}